\documentclass[a4paper,12pt]{article}
\usepackage{graphicx}
\usepackage{amsfonts}
\usepackage{amsthm}
\usepackage{pgfplots}
\usepackage{amsmath}
\usepackage{float}
\usepackage{enumerate}
\usepackage{multicol}
\usepackage{tikz}
\usepackage{todonotes}
\usepackage{caption}
\usetikzlibrary{positioning,decorations.pathmorphing,arrows.meta}

\title{\textbf{Monochromatic Components in Edge-Coloured Graphs with Large Minimum Degree}}
\author{Hannah Guggiari\thanks{Mathematical Institute, University of Oxford, Oxford, OX2 6GG, United Kingdom. Email:\texttt{\{guggiari,scott\}@maths.ox.ac.uk}}~ and Alex Scott\footnotemark[1] \thanks{Supported by a Leverhulme Trust Research Fellowship.}}
\date{\today}

\tikzset{main node/.style={circle,fill=black,draw,minimum width=4pt,inner sep=0pt}}

\DeclareCaptionType{Linear Program}

\newtheorem{theorem}{Theorem}[section]
\newtheorem{conjecture}[theorem]{Conjecture}
\newtheorem{lemma}[theorem]{Lemma}
\newtheorem{question}[theorem]{Question}

\newcommand{\lex}{\text{lex}}
\newcommand{\maxlex}{\text{maxlex}}

\begin{document}
\newpage
\maketitle
\begin{abstract}
\noindent
For every $n\in\mathbb{N}$ and $k\geq2$, Gy\'{a}rf\'{a}s showed that every $k$-edge-colouring of the complete graph on $n$ vertices contains a monochromatic connected component of order at least $\frac{n}{k-1}$. For $k\geq3$, Gy\'{a}rf\'{a}s and S\'{a}rk\"{o}zy proved that the complete graph can be replaced by a graph $G$ with $\delta(G)\geq(1-\varepsilon_k)n$ for some constant $\varepsilon_k$. In this paper, we show that the maximum possible value of $\varepsilon_3$ is $\frac16$. We will also show that $\varepsilon_k\le\frac{k-2}{k(k-1)}$ for infinitely many $k\geq3$. This disproves a conjecture of Gy\'{a}rfas and S\'{a}rk\"{o}zy.
\end{abstract}
\noindent
\section{Introduction}
Erd\H{o}s and Rado noted that, for any graph $G$, either $G$ or its complement is connected. This is equivalent to the statement that every 2-edge-colouring of a complete graph contains a monochromatic spanning tree. Gy\'{a}rf\'{a}s \cite{G77} extended this result to $k\geq3$ colours. He proved the following theorem.
\begin{theorem}[Gy\'{a}rf\'{a}s \cite{G77}]
\label{thm:Kn}
Fix $k\geq2$. In every $k$-edge-colouring of the complete graph on $n$ vertices, there exists a monochromatic component of order at least $\frac n{k-1}$.
\end{theorem}
\noindent
It should be noted that Mubayi and, independently, Liu, Morris and Prince found a short and elegant proof of this theorem (see \cite{G11}, Corollary 3.3).

The bound in Theorem \ref{thm:Kn} is sharp if $k-1$ is a prime power and $n$ is divisible by $(k-1)^2$. Consider the affine plane of order $k-1$ and colour the edges in the $i^\text{th}$ parallel class with colour $i$ for each $i\in[k]$. Every monochromatic component contains exactly $\frac n{k-1}$ vertices.

For $k=2$, it is easy to see that the conclusion of Theorem \ref{thm:Kn} does not hold for 2-colourings of the edges of a non-complete graph: if $xy$ is not an edge, then colour edges red if they are incident with $x$ and blue otherwise. However, Gy\'{a}rf\'{a}s and S\'{a}rk\"{o}zy \cite{GS12} proved the following theorem.
\begin{theorem}[Gy\'{a}rf\'{a}s and S\'{a}rk\"{o}zy \cite{GS12}]
\label{thm:2colours}
Let $G$ be a graph of order $n$ with $\delta(G)\geq\frac34n$. If the edges of $G$ are 2-coloured, then there exists a monochromatic component of order at least $\delta(G)+1$.
\end{theorem}
\noindent
Gy\'{a}rf\'{a}s and S\'{a}rk\"{o}zy also showed that the bounds in this theorem are best possible in the following two senses.

Firstly, we cannot reduce $\delta(G)$ below $\frac34n$. Suppose that $n$ is divisible by $4$ and partition the vertices into $4$ sets $V_1$, $V_2$, $V_3$ and $V_4$, each of order $\frac n4$. Colour the following edges red: the edges within each $V_i$, those between $V_1$ and $V_2$ and those between $V_3$ and $V_4$. Colour the edges between $V_1$ and $V_3$ blue and the edges between $V_2$ and $V_4$ blue. All monochromatic components have order $\frac n2$ and $\delta(G)=\frac34n-1$.

Secondly, the largest monochromatic component we can guarantee has order $\delta(G)+1$. Take the complete graph $K_n$ and let $X,Y\subset V(K_n)$ be disjoint vertex sets with $|X|=|Y|=n-\delta-1$ where $\frac n2-1\leq\delta\leq n-1$ (so $0\leq n-\delta-1\leq\frac n2$). Form the graph $G$ by removing all edges between $X$ and $Y$. Colour the edges incident to $X$ red, the edges incident to $Y$ blue and all other edges arbitrarily. Each vertex in $X\cup Y$ has degree exactly $\delta$ and all other vertices have degree $n-1$ so the minimum degree of $G$ is $\delta$. The largest monochromatic component in $G$ in either colour has order $\delta+1$.

For $k\geq3$, the situation is different. In this case, it is possible to remove some edges from a complete graph and still obtain a monochromatic component of order $\frac n{k-1}$ in every $k$-edge-colouring. Indeed, Gy\'{a}rf\'{a}s and S\'{a}rk\"{o}zy \cite{GS17} showed that the complete graph can be replaced by any graph $G$ with $\delta(G)\geq(1-\varepsilon_k)n$ for some constant $\varepsilon_k>0$. They made the following conjecture.
\begin{conjecture}[Gy\'{a}rf\'{a}s and S\'{a}rk\"{o}zy \cite{GS17}]
\label{conj:main}
Fix $k\geq3$. Let $G$ be any graph with $n$ vertices and $\delta(G)\geq\left(1-\frac{k-1}{k^2}\right)n$. If the edges of $G$ are $k$-coloured, then there exists a monochromatic component of order at least $\frac{n}{k-1}$.
\end{conjecture}
\noindent
Recently, there has been some progress towards Conjecture \ref{conj:main}. The best current general result was proved by DeBiasio, Krueger and S\'{a}rk\"{o}zy \cite{DKS18}.
\begin{theorem}[DeBiasio, Krueger and S\'{a}rk\"{o}zy \cite{DKS18}]
\label{thm:3072(k-1)^5}
Fix an integer $k\geq3$ and let $G$ be any graph of order $n$ with $\delta(G)\geq\left(1-\frac{1}{3072(k-1)^5}\right)n$. If the edges of $G$ are $k$-coloured, then there exists a monochromatic component of order at least $\frac{n}{k-1}$.
\end{theorem}
\noindent
For 3 colours, DeBiasio, Krueger and S\'{a}rk\"{o}zy \cite{DKS18} proved a stronger result.
\begin{theorem}[DeBiasio, Krueger and S\'{a}rk\"{o}zy \cite{DKS18}]
\label{thm:7/8}
Let $G$ be a graph of order $n$ with $\delta(G)\geq\frac78n$. In every 3-colouring of the edges of $G$, there exists a monochromatic component of order at least $\frac n2$.
\end{theorem}
\noindent
We will improve upon this result by proving the following theorem.
\begin{theorem}
\label{thm:56}
Let $G$ be a graph of order $n$ with $\delta(G)\geq\frac56n$ where the edges of $G$ have been $3$-coloured. Then $G$ has a monochromatic component with order at least $\frac n2$.
\end{theorem}
\noindent
We will then show that $\frac16$ is the largest possible value for $\varepsilon_3$ by proving the following theorem.
\begin{theorem}
\label{thm:counter3}
For every $n\in\mathbb{N}$, there exists a graph $G$ of order $n$ with $\delta(G)\geq\left\lfloor\frac56n\right\rfloor-2$ and a 3-colouring of the edges of $G$ such that every monochromatic component has order strictly less than $\frac n2$.
\end{theorem}
\noindent
This proves that Conjecture \ref{conj:main} is in fact false for every $n\in\mathbb{N}$ and $k=3$. We provide further counter-examples to the conjecture for infinitely many $k$ (and infinitely many $n$ for each such $k$) by proving the following theorem.
\begin{theorem}
\label{thm:counterk}
Let $k\ge3$ be a prime power.  For infinitely many $n\in\mathbb{N}$, there is a graph $G$ of order $n$ with 
$\delta(G)\geq\left(1-\frac{k-2}{k(k-1)}\right)n-2$ 
and a $k$-colouring of the edges of $G$ such that every monochromatic component has order strictly less than $\frac{n}{k-1}$.
\end{theorem}
\noindent
The paper is organised in the following way. In Sections \ref{sec:upper} and \ref{sec:lp}, we will prove Theorem \ref{thm:56}. First we will reduce the problem to two specific cases in Section \ref{sec:upper}. We will then explain how these two cases can be formulated as collections of linear programs and solved in Section \ref{sec:lp}. Then, in Section \ref{sec:lower}, we will prove Theorem \ref{thm:counter3}. Together these results show that $\varepsilon_3=\frac16$. Finally, in Section \ref{sec:counterk}, we will prove Theorem \ref{thm:counterk}, which shows that, for infinitely many $k\geq3$, we must have $\varepsilon_k\le \frac{k-2}{k(k-1)}$. This disproves Conjecture \ref{conj:main}.

\section{Reducing the upper bound to special cases}
\label{sec:upper}
To prove Theorem \ref{thm:56}, we will reduce the problem to one which can be written as a series of linear programs. We solve these linear programs by computer to show that the theorem is true.

We begin by proving a series of lemmas. Throughout this section, we will assume that the $3$ colours used to colour the edges are red, blue and green. If a vertex has no edges of a particular colour, then we do not regard it as being a component of that colour.
\begin{lemma}
\label{lem:allcolours}
Let $G$ be a graph of order $n$ with $\delta(G)\geq\frac56n$. Suppose that the edges of $G$ are $3$-coloured. If there exists a vertex $v$ which is not incident to edges of all $3$ colours, then there exists a monochromatic component of order at least $\frac n2$.
\end{lemma}
\begin{proof}
Colour the edges of $G$ red, blue and green. First consider the case where the vertex $v$ is only incident to edges of one colour, say red. As $\delta(G)\geq\frac56n$, the red component containing $v$ covers at least $\frac56n+1>\frac n2$ of the vertices of $G$.

Now consider the case where $v$ is only incident to edges of two colours, say red and blue. Let $R\subseteq V(G)$ be the vertices of the red component containing $v$ and $B\subseteq V(G)$ be the vertices of the blue component containing $v$. We may assume that $|R|<\frac n2$ and $|B|<\frac n2$.
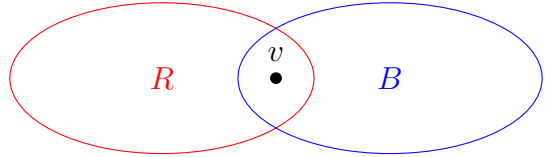
\begin{figure}[H]
\centering
\begin{tikzpicture}
\draw[red] (-1.5,0) ellipse (2cm and 1cm) node {$R$};
\draw[blue] (1.5,0) ellipse (2cm and 1cm) node {$B$};
\node[main node,label={$v$}] (v) at (0,0) {};
\end{tikzpicture}
\caption{The vertices of the red and blue components containing $v$.}
\label{fig:RBv}
\end{figure}
Let $x=|R\cap B|$, $r=|R\setminus B|$ and $b=|B\setminus R|$. Without loss of generality, we will assume that $r\geq b$. As $v\in R\cap B$ and $\delta(G)\geq\frac56n$, we find that $x>0$, $\frac5{12}n<r+x<\frac n2$ and $\frac13n<b+x\leq r+x$.

Suppose that $x<\frac16n$. As $r+x+b\geq\frac56n$ and $r\geq b$, we must have $2r\geq\frac56n-x$ giving $r>\frac13n$. It follows from $b+x>\frac13n$ that $b>\frac16n$. Any edges that are incident to both $R\setminus B$ and $B\setminus R$ must be green. As $\delta(G)\geq\frac56n$ and $r>\frac13n$, every pair of vertices in $B\setminus R$ must have a green neighbour in common in $R\setminus B$. As $b>\frac16n$, every vertex in $R\setminus B$ has a green neighbour in $B\setminus R$. Hence there is a green component covering all of $(R\cup B)\setminus(R\cap B)$. This green component has order at least $\frac16n+\frac13n\geq\frac n2$.

Now suppose that $x\geq\frac16n$. As $r+x+b\geq\frac56n$ and $r\geq b$, it follows that $r\geq\frac5{12}n-\frac12x$ and hence $|R|=r+x\geq\frac5{12}n+\frac12x\geq\frac n2$ giving a red component covering half of the vertices.
\end{proof}
\noindent
Given a graph $G$, the \textit{$t$-blow-up} $G'$ is the graph formed from $G$ by replacing each vertex with a copy of $K_t$ and each edge with a copy of $K_{t,t}$. If the edges of $G$ have been coloured, then the edges of $G'$ are coloured as follows:
\begin{itemize}
\item  edges between two $K_t$ are coloured according to the $3$-edge-colouring on $G$
\item edges within a $K_t$ are coloured arbitrarily
\end{itemize}
The graph $G'$ behaves like a larger version of $G$.
\begin{lemma}
\label{lem:blowup}
Fix $n>2$ and $t>1$. Let $G$ be a graph of order $n$. Colour the edges of $G$ so there is no monochromatic component covering half of the vertices. Let $G'$ be the $t$-blow-up of $G$. Then $G'$ has no monochromatic component covering half of its vertices and further $\delta(G')=t\delta(G)+t-1$.
\end{lemma}
\begin{proof}
Let $v_1,\dots,v_n$ be the vertices of $G$ and let $V_1,\dots,V_n$ be the corresponding copies of $K_t$ in $G'$. Fix a colour, say red. If $v_i$ has no red edges in $G$, then any red component containing a vertex of $V_i$ lies entirely within $V_i$ in $G'$. If instead $v_i$ is incident with a red edge in $G$, then the set $V_i$ is contained in some red component of $G'$. Furthermore $V_i$ and $V_j$ lie in the same red component if and only if $v_i$ and $v_j$ do. The remaining assertions are immediate.
\end{proof}
\begin{lemma}
\label{lem:2components}
Let $G$ be a graph of order $n$ with $\delta(G)\geq\frac56n$. Suppose that the edges of $G$ are $3$-coloured and there are exactly $2$ red components. Then there exists a monochromatic component of order at least $\frac n2$.
\end{lemma}
\begin{proof}
By Lemma \ref{lem:allcolours}, we are done unless every vertex is incident to edges of all three colours. Therefore every vertex is in a red component. As there are only two red components, one of them must cover at least $\frac n2$ of the vertices.
\end{proof}
\begin{lemma}
\label{lem:smallcomponents}
Let $G$ be a graph of order $n$ where the edges are $3$-coloured and there is no monochromatic component of order at least $\frac n2$. Suppose that $G$ has $r$ red, $b$ blue and $g$ green components and there exist red components $R_1$ and $R_2$ such that $|R_1|+|R_2|<\frac n2$. Then there is a graph $G'$ together with a $3$-edge-colouring such that there are $(r-1)$ red, $b$ blue and $g$ green components, $\frac{\delta(G')}{|G'|}\geq\frac{\delta(G)}{|G|}$ and there is no monochromatic component in $G'$ covering at least half of the vertices.
\end{lemma}
\begin{proof}
Note that, by Lemma \ref{lem:allcolours}, we may assume that every vertex is incident to all three colours.

Suppose first that there exists $v_1\in R_1$ and $v_2\in R_2$ such that $v_1v_2\notin E(G)$. Let $G'$ be a copy of $G$ with the additional red edge $v_1v_2$. Then $\delta(G')\geq\delta(G)$ and all components in $G'$ have the same number of vertices as they do in $G$ with the exception of $R_1$ and $R_2$ which form a single component in $G'$. As $|R_1|+|R_2|<\frac n2$, $G'$ contains no monochromatic component covering at least half of the vertices.

Now suppose that every vertex in $R_1$ is connected to every vertex in $R_2$. Fix $u\in R_1$ and $v\in R_2$ and, without loss of generality, assume that the edge $uv$ is blue. Let $G'$ be a $2$-blow-up of $G$ with same $3$-edge-colouring. The vertex $u$ in $G$ corresponds to vertices $u_1$ and $u_2$ in $G'$ and $v$ corresponds to $v_1$ and $v_2$. Change the colour of the edge $u_1v_1$ from blue to red.

The red components corresponding to $R_1$ and $R_2$ in $G'$ now form a single component of order $2(|R_1|+|R_2|)<n=\frac12|G'|$. The vertices $u_1$ and $v_1$ still lie in the same blue component via the blue path $u_1v_2u_2v_1$ and so changing the colour of the edge $u_1v_1$ does not change the orders of the other components in $G'$. By Lemma \ref{lem:blowup}, we have $\delta(G')=2\delta(G)+1$ and $G'$ does not contain a monochromatic component of order at least $\frac12|G'|$. As $|G'|=2|G|$, it follows that $\frac{\delta(G')}{|G'|}>\frac{\delta(G)}{|G|}$.
\end{proof}
\noindent
Lemma \ref{lem:2components} and Lemma \ref{lem:smallcomponents} allow us to make the following assumption: in each colour, $G$ either has $3$ components or $4$ components each of order exactly $\frac n4$.
\begin{lemma}
\label{lem:onlyonecolourwith4}
Let $G$ be a graph of order $n$ where the edges of $G$ are $3$-coloured. Suppose that $G$ has no monochromatic component of order at least $\frac n2$. If, in two colours, there are $4$ components of order exactly $\frac n4$ in that colour, then $\delta(G)<\frac56n$.
\end{lemma}
\begin{proof}
Without loss of generality, suppose $G$ has 4 red components and 4 blue components, each with order exactly $\frac n4$. By Lemma \ref{lem:allcolours}, every vertex lies in components of all three colours. Lemma \ref{lem:2components} tells us that there must be at least $3$ green components. The smallest green component has order at most $\frac n3$. Choose a vertex $v$ in the smallest green component. Then we find that $d(v)\leq(\frac n3-1)+2(\frac n4-1)<\frac56n$.
\end{proof}
\noindent
The above lemmas allow us to make the following assumptions about $G$:
\begin{itemize}
\item $G$ is a graph with $n$ vertices, minimum degree at least $\frac56n$ and no monochromatic component of order at least $\frac n2$.
\item Every vertex is incident to an edge of every colour (Lemma \ref{lem:allcolours}).
\item There are either $3$ components in each colour or there are $3$ components in two colours and $4$ components of order $\frac n4$ in the third colour (Lemmas \ref{lem:2components}, \ref{lem:smallcomponents} and \ref{lem:onlyonecolourwith4}).
\end{itemize}

\section{Linear programs for the upper bound}
\label{sec:lp}
In Section \ref{sec:upper}, we reduced the proof of Theorem \ref{thm:56} to the following two cases:
\begin{enumerate}
\item Every vertex is incident to an edge of every colour. There are $3$ components of each colour.
\item Every vertex is incident to an edge of every colour. There are $3$ components in two of the colours (without loss of generality, red and blue) and $4$ components of order exactly $\frac n4$ in the third colour (without loss of generality, green).
\end{enumerate}
We now formulate these cases as collections of linear programs. More details about the code used to implement these linear programs may be found in Appendix \ref{sec:lpcode}.

\subsection{Three components in each colour}
We begin by considering the first case where there are $3$ components in each of the three colours. Let the vertex sets of the red components be $R_i$, the blue components be $B_j$ and the green components be $G_k$ for $i,j,k\in[3]$. We know that every vertex $v$ of $G$ lies in the intersection $R_i\cap B_j\cap G_k$ for some $i,j,k\in[3]$ and so $d(v)\leq|R_i\cup B_j\cup G_k|-1$ with equality if $v$ is adjacent to every vertex in $R_i\cup B_j\cup G_k$. Proving Theorem \ref{thm:56} for the first case is equivalent to showing that Question \ref{qu:n} has an answer of $\alpha<\frac56$.
\begin{question}
\label{qu:n}
What is the maximum value of $\alpha$ such that the following conditions can hold simultaneously?
\begin{enumerate}
\item $|R_i|,|B_j|,|G_k|<\frac n2\qquad\forall i,j,k\in[3]$
\item $d(v)\geq\alpha n\qquad\forall v\in V(G)$
\end{enumerate}
\end{question}
\noindent
For any $v\in R_i\cap B_j\cap G_k$, the addition of any missing edges between $v$ and $R_i\cup B_j\cup G_k$ will not change the number of vertices in each monochromatic component but may increase $\delta(G)$ and hence $\alpha$. Therefore we may assume that $d(v)=|R_i\cup B_j\cup G_k|-1$ for every vertex $v\in R_i\cap B_j\cap G_k$ (although it is important to note that $R_i\cap B_j\cap G_k$ may be empty).

We can avoid dependence on $n$ by rescaling. Let $x_{ijk}=\frac1n|R_i\cap B_j\cap G_k|$ for each $i,j,k\in[3]$. For fixed $i\in[3]$, we find that $|R_i|=n\sum_{j=1}^3\sum_{k=1}^3x_{ijk}$ (and similarly for $|B_j|$ and $|G_k|$) and, for fixed $i,j,k\in[3]$, we have
\begin{equation*}
\frac1n|R_i\cup B_j\cup G_k|=\sum_{\substack{i',~j',~k'\\i'=i\text{ or }j'=j\text{ or }k'=k}}x_{i'j'k'}.
\end{equation*}
Using this notation and dividing through by $n$, the first condition in Question \ref{qu:n} becomes:
\begin{align*}
\sum_{j=1}^3\sum_{k=1}^3x_{ijk}&<\frac 12\qquad\qquad\forall i\in[3]\\
\sum_{i=1}^3\sum_{k=1}^3x_{ijk}&<\frac 12\qquad\qquad\forall j\in[3]\\
\sum_{i=1}^3\sum_{j=1}^3x_{ijk}&<\frac 12\qquad\qquad\forall k\in[3]
\end{align*}
In linear programs, the condition statements must consist of weak, rather than strict, inequalities in order to guarantee that the space of feasible solutions is closed and, if this space is non-empty, that an optimal solution exists. Relaxing the above conditions to allow equality will increase the space of feasible solutions. As the maximum value of $\alpha$ for our original problem will be at most the maximum value of $\alpha$ for the relaxed problem, showing that $\alpha<\frac56$ for the relaxed problem is sufficient to prove Theorem \ref{thm:56}.

The second condition holds whenever $R_i\cap B_j\cap G_k$ is non-empty. We obtain:
\begin{equation*}
\begin{cases}
\displaystyle\sum_{\substack{i',~j',~k'\\i'=i\text{ or }j'=j\text{ or }k'=k}}x_{i'j'k'}\geq\alpha+\frac1n&\text{if }R_i\cap B_j\cap G_k\neq\emptyset\\
\displaystyle\qquad\quad\qquad x_{ijk}=0&\text{otherwise.}
\end{cases}
\end{equation*}
We would like to remove the dependence on $n$ completely. We therefore relax the second condition by removing the $\frac1n$; we will obtain an upper bound of $\alpha\leq\frac56$ which is sufficient to prove Theorem \ref{thm:56}.

We encode whether $R_i\cap B_j\cap G_k$ is empty with an additional variable $y_{ijk}$ by setting:
\begin{equation*}
y_{ijk}=
\begin{cases}
1&\text{if }R_i\cap B_j\cap G_k\neq\emptyset\\
0&\text{otherwise.}
\end{cases}
\end{equation*}
The variables $y_{ijk}$ represent the pattern of intersections. For a fixed intersection pattern, an upper bound on $\alpha$ in Question \ref{qu:n} can be found by solving Linear Program \ref{lp:333}.

\begin{Linear Program}[ht]
\centering
\begin{align*}
&\textbf{maximise }&\alpha&&&\\
&\textbf{subject to }&\sum_{i=1}^3\sum_{j=1}^3\sum_{k=1}^3x_{ijk}&=1&&\\
&&\sum_{j=1}^3\sum_{k=1}^3x_{ijk}&\leq\frac12&\forall i&\in[3]\\
&&\sum_{i=1}^3\sum_{k=1}^3x_{ijk}&\leq\frac12&\forall j&\in[3]\\
&&\sum_{i=1}^3\sum_{j=1}^3x_{ijk}&\leq\frac12&\forall k&\in[3]\\
&&\sum_{\substack{i',~j',~k'\\i'=i\text{ or }j'=j\text{ or }k'=k}}x_{i'j'k'}-\alpha&\geq0&\forall i,j,k&\in[3]\text{ such that }y_{ijk}=1\\
&&x_{ijk}&=0&\forall i,j,k&\in[3]\text{ such that } y_{ijk}=0\\
&&x_{ijk}&\geq0&\forall i,j,k&\in[3]
\end{align*}
\caption{3 red, 3 blue and 3 green components.}
\label{lp:333}
\end{Linear Program}
\noindent
We have already assumed that there are $3$ components of each colour or else we would be done by Lemma \ref{lem:2components}. Therefore, we only need to consider intersection patterns that satisfy the following condition: for each $i\in[3]$, there exists $j,k\in[3]$ such that $y_{ijk}\neq0$ (and similarly for $j$ and $k$).

Using a computer, we ran the linear program for all valid intersection patterns (roughly $2^{27}$ linear programs were run) and found the maximum overall value of $\alpha$ and the optimal solutions corresponding to this value of $\alpha$. The maximum value was $\alpha=\frac56$.

All of the optimal solutions were subgraphs of the following graph: the vertices are divided equally into six vertex sets. The edges within vertex sets are coloured arbitrarily and the edges between vertex classes are coloured:
\begin{itemize}
\item Red: edges between $V_1$ and $V_5$ and edges between $V_2$ and $V_4$
\item Blue: edges between $V_2$ and $V_6$ and edges between $V_3$ and $V_5$
\item Green: edges between $V_1$ and $V_3$ and edges between $V_4$ and $V_6$
\item Red or blue: edges between $V_1$ and $V_6$ and edges between $V_3$ and $V_4$
\item Red or green: edges between $V_2$ and $V_3$ and edges between $V_5$ and $V_6$
\item Blue or green: edges between $V_1$ and $V_2$ and edges between $V_4$ and $V_5$.
\end{itemize}

This graph is shown in Figure \ref{fig:opt}. In the figure, each circle represents a vertex set, each containing approximately the same number of vertices. The lines show the colours of the edges between vertex sets. Where the line between two circles is striped, the edges between these two vertex sets may be either of the two colours indicated (or a mixture of the two).

\begin{figure}[!h]
\centering
\begin{tikzpicture}
\draw (3.5,0) node (v1) {};
\draw (1.5,-3) node (v2) {};
\draw (-1.5,-3) node (v3) {};
\draw (-3.5,0) node (v4) {};
\draw (-1.5,3) node (v5) {};
\draw (1.5,3) node (v6) {};

\draw[red, ultra thick] (v1) to (v5);
\draw[red, ultra thick] (v1) to (v6);
\draw[red, ultra thick] (v5) to (v6);
\draw[red, ultra thick] (v2) to (v3);
\draw[red, ultra thick] (v2) to (v4);
\draw[red, ultra thick] (v3) to (v4);

\draw[blue, ultra thick] (v1) to (v2);
\draw[blue, ultra thick, dash pattern=on 5pt off 5pt] (v1) to (v6);
\draw[blue, ultra thick] (v2) to (v6);
\draw[blue, ultra thick, dash pattern=on 5pt off 5pt] (v3) to (v4);
\draw[blue, ultra thick] (v3) to (v5);
\draw[blue, ultra thick] (v4) to (v5);

\draw[green, ultra thick, dash pattern=on 5pt off 5pt] (v1) to (v2);
\draw[green, ultra thick] (v1) to (v3);
\draw[green, ultra thick, dash pattern=on 5pt off 5pt] (v2) to (v3);
\draw[green, ultra thick, dash pattern=on 5pt off 5pt] (v4) to (v5);
\draw[green, ultra thick] (v4) to (v6);
\draw[green, ultra thick, dash pattern=on 5pt off 5pt] (v5) to (v6);
	
\draw[thick, fill=white] (v1) circle (0.5cm) node {$V_1$};
\draw[thick, fill=white] (v2) circle (0.5cm) node {$V_2$};
\draw[thick, fill=white] (v3) circle (0.5cm) node {$V_3$};
\draw[thick, fill=white] (v4) circle (0.5cm) node {$V_4$};
\draw[thick, fill=white] (v5) circle (0.5cm) node {$V_5$};
\draw[thick, fill=white] (v6) circle (0.5cm) node {$V_6$};
\end{tikzpicture}
\caption{Optimal solutions of Linear Program \ref{lp:333} are subgraphs of this graph.}
\label{fig:opt}
\end{figure}
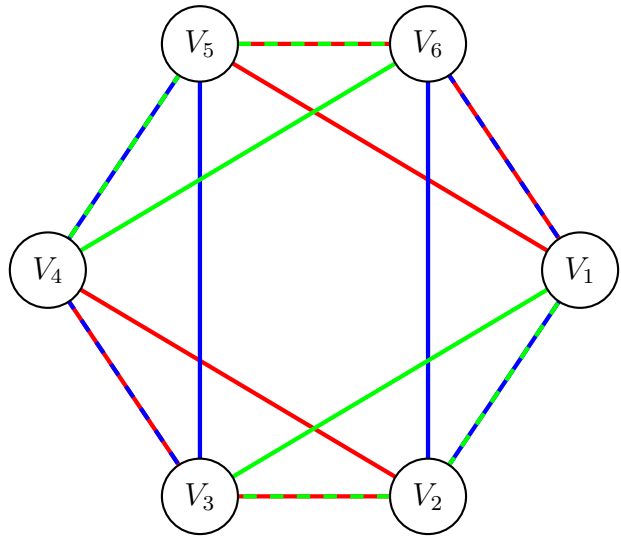

\subsection{Three components in two colours; four in the third}
Now we consider the case where there are $3$ red, $3$ blue and $4$ green components and each green component has order exactly $\frac n4$. The set-up is very similar to the case above. For a given intersection pattern $y_{ijk}$, the only difference in the linear program is the condition that each green component has size exactly $\frac14$ (rather than just being at most $\frac12$).
\begin{Linear Program}[H]
\centering
\begin{align*}
&\textbf{maximise }&\alpha&&&\\
&\textbf{subject to }&\sum_{i=1}^3\sum_{j=1}^3\sum_{k=1}^4x_{ijk}&=1&&\\
&&\sum_{j=1}^3\sum_{k=1}^4x_{ijk}&\leq\frac12&\forall i&\in[3]\\
&&\sum_{i=1}^3\sum_{k=1}^4x_{ijk}&\leq\frac12&\forall j&\in[3]\\
&&\sum_{i=1}^3\sum_{j=1}^3x_{ijk}&=\frac14&\forall k&\in[4]\\
&&\sum_{\substack{i',~j',~k'\\i'=i\text{ or }j'=j\text{ or }k'=k}}x_{i'j'k'}-\alpha&\geq0&\forall i,j&\in[3],k\in[4]\text{ such that }y_{ijk}=1\\
&&x_{ijk}&=0&\forall i,j&\in[3],k\in[4]\text{ such that }y_{ijk}=0\\
&&x_{ijk}&\geq0&\forall i,j&\in[3],k\in[4]
\end{align*}
\caption{3 red, 3 blue and 4 green components.}
\label{lp:334}
\end{Linear Program}
\noindent
As above, we only considered intersection patterns where, for each $i\in[3]$, there exists $j\in[3],k\in[4]$ such that $y_{ijk}\neq0$ (and similarly for $j$ and $k$). Using a computer, we ran the linear program for all valid intersection patterns ($\approx2^{30}$ linear programs were run) and found that the maximum overall value of $\alpha$ was $\frac34$. As $\frac34$ is strictly smaller than $\frac56$, we may conclude that Theorem \ref{thm:56} holds in the case where there are four components in one colour and three components in the other two colours.

\section{Proof of lower bound}
\label{sec:lower}
In this section, we give constructions to prove the lower bound given in Theorem \ref{thm:counter3}.
\begin{proof}[Proof of Theorem \ref{thm:counter3}]
We will give a separate construction for each residue class modulo 6, beginning with the case where $n=6q$ for some $q\in\mathbb{N}$. In each construction, the edges inside vertex classes can be coloured arbitrarily and so we will only specify the colours of edges between vertex classes.
\\
\\
\textit{Case: $n=6q$.} We will construct a graph $G$ of order $n$ with $\delta(G)=\left\lfloor\frac56n\right\rfloor-2$. We will also show that there is a $3$-colouring of the edges of $G$ such that every monochromatic component has order strictly less than $\frac n2$. The colours will be red, blue and green.
	
Partition the vertices into $6$ sets $V_1,\dots,V_6$ with the following sizes:
\begin{itemize}
\item $|V_1|=|V_3|=|V_5|=q-1$
\item $|V_2|=|V_4|=|V_6|=q+1$.
\end{itemize}
Observe that $|V_1|+\dots+|V_6|=6q=n$. There are no edges between:
\begin{itemize}
\item $V_1$ and $V_4$
\item $V_2$ and $V_5$
\item $V_3$ and $V_6$.
\end{itemize}
All other edges are present (including all edges inside vertex classes). This means that each vertex in $V_1\cup V_3\cup V_5$ has degree $5q-2$ and each vertex in $V_2\cup V_4\cup V_6$ has degree $5q$. Therefore $\delta(G)=5q-2=\left\lfloor\frac56n\right\rfloor-2$ as required.
	
It remains to construct a $3$-colouring of the edges in which every monochromatic component has order strictly less than $\frac n2$. We colour edges between vertex classes as follows:
\begin{itemize}
\item Red: edges between $V_2$ and $V_6$ and edges between $V_3$, $V_4$ and $V_5$
\item Blue: edges between $V_2$ and $V_4$ and edges between $V_1$, $V_5$ and $V_6$
\item Green: edges between $V_4$ and $V_6$ and edges between $V_1$, $V_2$ and $V_3$.
\end{itemize}
$G$ is the graph defined by this colouring (see Figure \ref{fig:counter0}).
\begin{figure}[!h]
\centering
\begin{tikzpicture}
\draw (3.5,0) node (v1) {};
\draw (1.5,-3) node (v2) {};
\draw (-1.5,-3) node (v3) {};
\draw (-3.5,0) node (v4) {};
\draw (-1.5,3) node (v5) {};
\draw (1.5,3) node (v6) {};

\draw[red, ultra thick] (v2) to (v6);
\draw[red, ultra thick] (v3) to (v4);
\draw[red, ultra thick] (v3) to (v5);
\draw[red, ultra thick] (v4) to (v5);
\draw[blue, ultra thick] (v2) to (v4);
\draw[blue, ultra thick] (v1) to (v5);
\draw[blue, ultra thick] (v1) to (v6);
\draw[blue, ultra thick] (v5) to (v6);
\draw[green, ultra thick] (v4) to (v6);
\draw[green, ultra thick] (v1) to (v2);
\draw[green, ultra thick] (v1) to (v3);
\draw[green, ultra thick] (v2) to (v3);

\draw[thick, fill=white] (v1) circle (0.5cm) node {$V_1$};
\draw[thick, fill=white] (v2) circle (0.5cm) node {$V_2$};
\draw[thick, fill=white] (v3) circle (0.5cm) node {$V_3$};
\draw[thick, fill=white] (v4) circle (0.5cm) node {$V_4$};
\draw[thick, fill=white] (v5) circle (0.5cm) node {$V_5$};
\draw[thick, fill=white] (v6) circle (0.5cm) node {$V_6$};
\end{tikzpicture}
\caption{The graph $G$ for $n=6q$.}
\label{fig:counter0}
\end{figure}
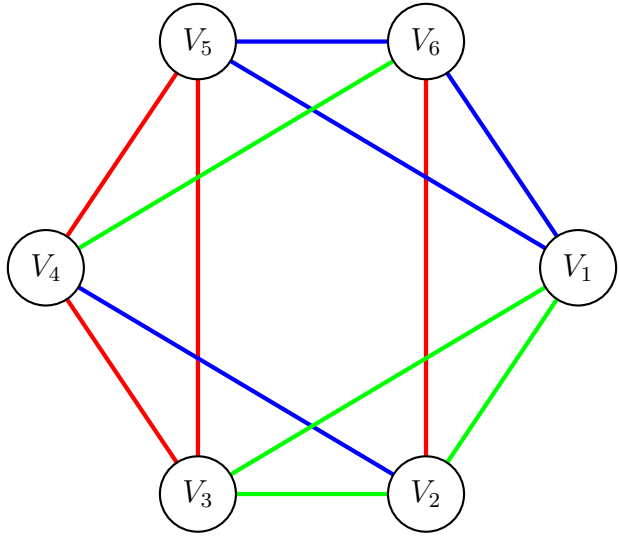
	
In this colouring, the two largest red components (the components which contain contain vertices from at least two vertex classes) have orders $2q+2$ and $3q-1$. Similarly, the two largest blue components have orders $2q+2$ and $3q-1$ and the two largest green components have orders $2q+2$ and $3q-1$. As $\frac n2=3q$, all monochromatic components have order strictly less than $\frac n2$ as required.
	
For the remaining residue classes modulo $6$, we construct similar graphs with $\delta(G)=\left\lfloor\frac56n\right\rfloor-c$ for some constant $c\in\{1,2\}$ that depends on the residue class.
\\
\\
\textit{Case: $n=6q+1$.} Partition the vertices into $6$ sets.
\begin{itemize}
\item $|V_1|=|V_2|=|V_3|=q+1$
\item $|V_4|=q$
\item $|V_5|=|V_6|=q-1$.
\end{itemize}
We colour edges in the same way as in Figure \ref{fig:counter0} to get the graph $G$. (Note that the number of vertices in each set is different from the case where $n\equiv0\mod6$.) The largest monochromatic component has order $3q<\frac n2$ and $\delta(G)=5q-1=\left\lfloor\frac56n\right\rfloor-1$.
\\
\\
\textit{Case: $n=6q+2$.} Partition the vertices into $8$ sets.
\begin{multicols}{2}
\begin{itemize}
\item $|V_1|=q-4$
\item $|V_2|=|V_3|=|V_5|=|V_6|=q+1$
\item $|V_4|=q-2$
\item $|V_7|=|V_8|=2$.
\end{itemize}
\end{multicols}
\noindent
We colour edges between vertex classes as follows to get the graph $G$ (see Figure \ref{fig:counter2}):
\begin{itemize}
\item Red: edges between $V_2$ and $V_6$, edges between $V_3$, $V_4$ and $V_5$ and edges between $V_7$ and $V_1\cup V_2\cup V_6$
\item Blue: edges between $V_2$ and $V_4$, edges between $V_3$ and $V_7$ and edges between $V_1$, $V_5$, $V_6$ and $V_8$
\item Green: edges between $V_1$, $V_2$ and $V_3$, edges between $V_8$ and $V_2\cup V_3$, edges between $V_4$ and $V_6$ and edges between $V_5$ and $V_7$.
\end{itemize}
Note that this example still holds if the colours of some or all of the edges between certain vertex classes are changed - edges between $V_1$ and $V_8$ may also be green and edges between $V_1$ and $V_2\cup V_6$ may also be red.
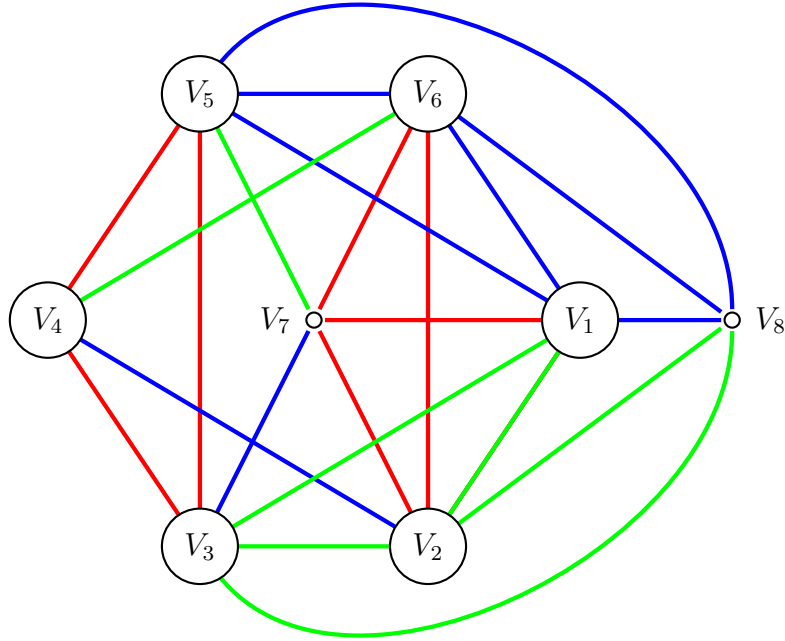
\begin{figure}[!h]
\centering
\begin{tikzpicture}
\draw[use as bounding box, white] (-4,-4.2) -| (4,-4.2) |- (4,4.2) -| (-4,4.2);
\draw (3.5,0) node (v1) {};
\draw (1.5,-3) node (v2) {};
\draw (-1.5,-3) node (v3) {};
\draw (-3.5,0) node (v4) {};
\draw (-1.5,3) node (v5) {};
\draw (1.5,3) node (v6) {};
\draw (0,0) node (v7) {};
\draw (5.5,0) node (v8) {};

\draw[red, ultra thick] (v7) to (v1);
\draw[red, ultra thick] (v7) to (v6);
\draw[red, ultra thick] (v7) to (v2);
\draw[red, ultra thick] (v6) to (v2);
\draw[red, ultra thick] (v4) to (v5);
\draw[red, ultra thick] (v4) to (v3);
\draw[red, ultra thick] (v5) to (v3);
\draw[red, ultra thick] (v1) to (v2);
\draw[blue, ultra thick] (v8) to (v1);
\draw[blue, ultra thick] (v8) to (v6);
\draw[blue, ultra thick] (v8) to[out=90,in=60] (v5);
\draw[blue, ultra thick] (v1) to (v6);
\draw[blue, ultra thick] (v1) to (v5);
\draw[blue, ultra thick] (v6) to (v5);
\draw[blue, ultra thick] (v7) to (v3);
\draw[blue, ultra thick] (v4) to (v2);
\draw[green, ultra thick] (v8) to (v2);
\draw[green, ultra thick] (v8) to[out=-90,in=-60] (v3);
\draw[green, ultra thick] (v1) to (v2);
\draw[green, ultra thick] (v1) to (v3);
\draw[green, ultra thick] (v2) to (v3);
\draw[green, ultra thick] (v4) to (v6);
\draw[green, ultra thick] (v7) to (v5);

\draw[thick, fill=white] (v1) circle (0.5cm) node {$V_1$};
\draw[thick, fill=white] (v2) circle (0.5cm) node {$V_2$};
\draw[thick, fill=white] (v3) circle (0.5cm) node {$V_3$};
\draw[thick, fill=white] (v4) circle (0.5cm) node {$V_4$};
\draw[thick, fill=white] (v5) circle (0.5cm) node {$V_5$};
\draw[thick, fill=white] (v6) circle (0.5cm) node {$V_6$};
\draw[thick, fill=white] (v7) circle (0.1cm) node[label={left:$V_7$}] {};
\draw[thick, fill=white] (v8) circle (0.1cm) node[label={right:$V_8$}] {};
\end{tikzpicture}
\caption{The graph $G$ for $n=6q+2$.}
\label{fig:counter2}
\end{figure}
	
We see that the largest monochromatic component has order $3q=\frac n2-1$ and $\delta(G)=5q=\left\lfloor\frac56n\right\rfloor-1$.
\\
\\
\textit{Case: $n=6q+3$.} Partition the vertices into $6$ sets.
\begin{itemize}
\item $|V_1|=|V_2|=|V_3|=q+1$
\item $|V_4|=|V_5|=|V_6|=q$.
\end{itemize}
We colour edges between the vertex classes as in Figure \ref{fig:counter0} to get the graph $G$. We see that the largest monochromatic component has order $3q+1<\frac n2$ and $\delta(G)=5q+1=\left\lfloor\frac56n\right\rfloor-1$.
\\
\\
\textit{Case: $n=6q+4$}
Partition the vertices into $8$ sets.
\begin{multicols}{2}
\begin{itemize}
\item $|V_1|=q-3$
\item $|V_2|=|V_3|=|V_5|=|V_6|=q+1$
\item $|V_4|=q-1$
\item $|V_7|=|V_8|=2$.
\end{itemize}
\end{multicols}
We colour the edges between the vertex classes as in Figure \ref{fig:counter2} to get the graph $G$. We see that the largest monochromatic component has order $3q+1<\frac n2$ and $\delta(G)=5q+2=\left\lfloor\frac56n\right\rfloor-1$.
\\
\\
\textit{Case: $n=6q+5$.} Partition the vertices into $8$ sets.
\begin{multicols}{2}
\raggedcolumns
\interlinepenalty=10000
\begin{itemize}
\item $|V_1|=q-1$
\item $|V_2|=|V_3|=|V_5|=|V_6|=q+1$
\item $|V_4|=q$
\item $|V_7|=1$
\item $|V_8|=1$
\end{itemize}
\end{multicols}
We colour the edges between vertex classes as in Figure \ref{fig:counter2} to get the graph $G$. We see that the largest monochromatic component has order $3q+2<\frac n2$ and $\delta(G)=5q+3=\left\lfloor\frac56n\right\rfloor-1$.
\end{proof}
\noindent
It is worth noting that, as $n\rightarrow\infty$, the graphs shown in Figure \ref{fig:counter0} (corresponding to the cases where $n\equiv0,1,3\mod6$) and Figure \ref{fig:counter2} (where $n\equiv2,4,5\mod6$) are close to the graph in Figure \ref{fig:limit}.
\begin{figure}[!h]
\centering
\begin{tikzpicture}
\draw (3.5,0) node (v1) {};
\draw (1.5,-3) node (v2) {};
\draw (-1.5,-3) node (v3) {};
\draw (-3.5,0) node (v4) {};
\draw (-1.5,3) node (v5) {};
\draw (1.5,3) node (v6) {};

\draw[red, ultra thick] (v4) to (v3);
\draw[red, ultra thick] (v4) to (v5);
\draw[red, ultra thick] (v3) to (v5);
\draw[red, ultra thick] (v2) to (v6);
\draw[blue, ultra thick] (v4) to (v2);
\draw[blue, ultra thick] (v5) to (v6);
\draw[blue, ultra thick] (v5) to (v1);
\draw[blue, ultra thick] (v6) to (v1);
\draw[green, ultra thick] (v4) to (v6);
\draw[green, ultra thick] (v2) to (v1);
\draw[green, ultra thick] (v2) to (v3);
\draw[green, ultra thick] (v3) to (v1);

\draw[thick, fill=white] (v1) circle (0.5cm) node {$V_1$};
\draw[thick, fill=white] (v2) circle (0.5cm) node {$V_2$};
\draw[thick, fill=white] (v3) circle (0.5cm) node {$V_3$};
\draw[thick, fill=white] (v4) circle (0.5cm) node {$V_4$};
\draw[thick, fill=white] (v5) circle (0.5cm) node {$V_5$};
\draw[thick, fill=white] (v6) circle (0.5cm) node {$V_6$};
\end{tikzpicture}
\caption{The limit graph of G for for all residue classes modulo 6.}
\label{fig:limit}
\end{figure}
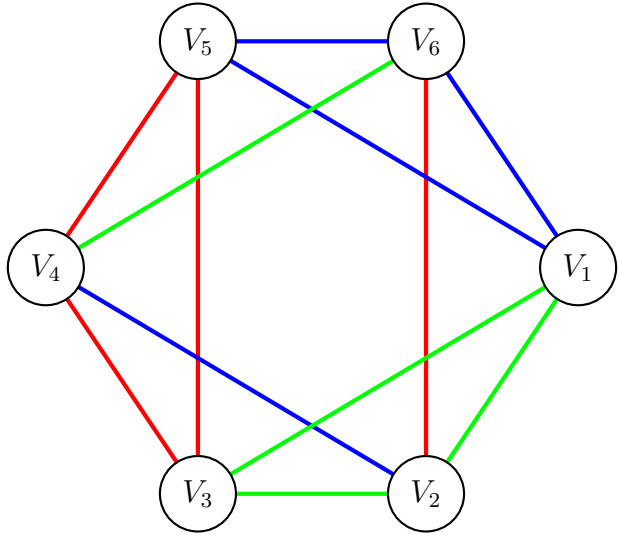
This graph has six vertex classes of equal size with the edges are coloured as stated below and is one of the optimal cases found in Section \ref{sec:lp}.
\begin{itemize}
\item Red: edges between $V_2$ and $V_6$ and edges between $V_3$, $V_4$ and $V_5$.
\item Blue: edges between $V_2$ and $V_4$ and edges between $V_1$, $V_5$ and $V_6$.
\item Green: edges between $V_4$ and $V_6$ and edges between $V_1$, $V_2$ and $V_3$.
\item Red, blue or green: edges within vertex classes.
\end{itemize}

The graphs given in the proof of Theorem \ref{thm:counter3} are by no means unique. For each residue class, we can find other graphs $G'$ with $\delta(G')=\left\lfloor\frac56n\right\rfloor-c$ for some small $c$ that have no monochromatic component covering half of the vertices. Indeed, for each of the optimal graphs found in Section \ref{sec:lp}, it is possible to find some such $G'$ which is close to it as $n\rightarrow\infty$.

It is also worth remarking that the constructions used in the proof of Theorem \ref{thm:counter3} are sharp in some cases. Suppose that there is a graph $G$ on $n$ vertices with no monochromatic component of order $\frac n2$ and $\delta(G)=\frac56n-a$ for some $a$. Theorem \ref{thm:56} tells us that $a>0$. If $G'$ is a $t$-blow-up of $G$, then $G'$ has no monochromatic component of order $\frac{tn}2$. By Lemma \ref{lem:blowup}, we find
\begin{align*}
\delta(G')&=t\delta(G)+(t-1)\\
&=\frac56nt-at+(t-1)
\end{align*}
and Theorem \ref{thm:56} tells us that $\delta(G')<\frac56tn$. Combining these inequalities gives $a>1-\frac1t$ for every $t\in\mathbb{N}$ and so $a\geq1$. Hence it follows that $\delta(G)\leq\left\lfloor\frac56n-1\right\rfloor=\left\lfloor\frac56n\right\rfloor-1$. In the proof of Theorem \ref{thm:counter3}, the graphs given for $n\equiv1,2,3,4,5\mod6$ each had minimum degree $\left\lfloor\frac56n\right\rfloor-1$.

\section{Counterexamples for infinitely many $k$}
\label{sec:counterk}
In this section, we will disprove Conjecture \ref{conj:main} for all prime powers $k\geq3$. (We initially constructed a counterexample only for $k=3$.  Building on an idea of DeBiasio \cite{D19}, we subsequently extended the construction to infinitely many $k$\footnote{DeBiasio and Krueger independently went on to produce a set of counterexamples \cite{DK20}.}.)

The affine plane of order $k$ is the decomposition of the edges of $K_{k^2}$ into $k+1$ families of $k$ vertex-disjoint $K_k$.  Affine planes exist whenever $k$ is a prime power.  Given an affine plane, we may colour the edges of each family with a different colour. Figure \ref{fig:affine} shows the affine plane for $k=3$, with the vertex families being:
\begin{itemize}
\item Red: $\{v_1, v_2, v_3\}$, $\{v_4, v_5, v_6\}$ and $\{v_7, v_8, v_9\}$
\item Blue: $\{v_1, v_4, v_7\}$, $\{v_2, v_5, v_8\}$ and $\{v_3, v_6, v_9\}$
\item Green: $\{v_1, v_6, v_8\}$, $\{v_2, v_4, v_9\}$ and $\{v_3, v_5, v_7\}$
\item Yellow: $\{v_1, v_5, v_9\}$, $\{v_2, v_6, v_7\}$ and $\{v_3, v_4, v_8\}$.
\end{itemize}
Note that, for clarity, not all edges between vertices in the same vertex set are shown. For instance, the red edges between $v_1$ and $v_3$ are missing from Figure \ref{fig:affine}.
\begin{figure}[!htbp]
\centering
\begin{tikzpicture}
\node[main node, label={90:$v_1$}] (v1) at (-2,2) {};
\node[main node] (v2) at (0,2) {};
\node[main node, label={90:$v_3$}] (v3) at (2,2) {};
\node[main node, label={180:$v_4$}] (v4) at (-2,0) {};
\node[main node] (v5) at (0,0) {};
\node[main node, label={0:$v_6$}] (v6) at (2,0) {};
\node[main node, label={-90:$v_7$}] (v7) at (-2,-2) {};
\node[main node] (v8) at (0,-2) {};
\node[main node, label={-90:$v_9$}] (v9) at (2,-2) {};

\draw[red, ultra thick] (v1) to (v2);
\draw[red, ultra thick] (v2) to (v3);
\draw[red, ultra thick] (v4) to (v5);
\draw[red, ultra thick] (v5) to (v6);
\draw[red, ultra thick] (v7) to (v8);
\draw[red, ultra thick] (v8) to (v9);

\draw[blue, ultra thick] (v1) to (v4);
\draw[blue, ultra thick] (v4) to (v7);
\draw[blue, ultra thick] (v2) to (v5);
\draw[blue, ultra thick] (v5) to (v8);
\draw[blue, ultra thick] (v3) to (v6);
\draw[blue, ultra thick] (v6) to (v9);

\draw[green, ultra thick] (v2) to (v4);
\draw[green, ultra thick] (v2) to[out=60,in=150] (2.5,2.5) to[out=-30,in=30](v9);
\draw[green, ultra thick] (v3) to (v5);
\draw[green, ultra thick] (v5) to (v7);
\draw[green, ultra thick] (v1) to[out=-150,in=150] (-2.5,-2.5) to[out=-30,in=-120] (v8);
\draw[green, ultra thick] (v6) to (v8);

\draw[yellow, ultra thick] (v1) to (v5);
\draw[yellow, ultra thick] (v5) to (v9);
\draw[yellow, ultra thick] (v2) to (v6);
\draw[yellow, ultra thick] (v2) to[out=120,in=30] (-2.5,2.5) to[out=-150,in=150] (v7);
\draw[yellow, ultra thick] (v3) to[out=-30,in=30] (2.5,-2.5) to[out=-150,in=-60] (v8);
\draw[yellow, ultra thick] (v4) to (v8);

\node[main node, label={90:$v_2$}] () at (v2) {};
\node[main node, label={135:$v_5$}] () at (v5) {};
\node[main node, label={-90:$v_8$}] () at (v8) {};
\end{tikzpicture}
\caption{Affine plane when $k=3$.}
\label{fig:affine}
\end{figure}
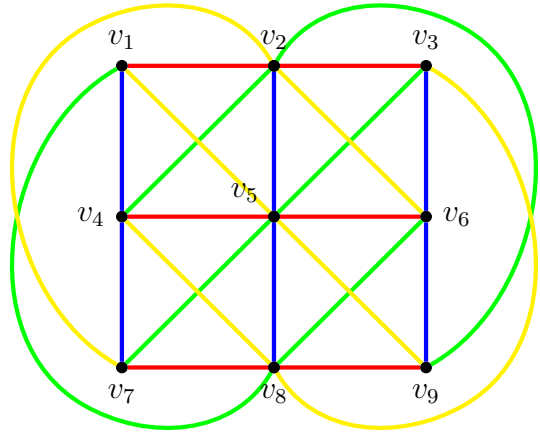

Let $\mathcal{F}$ be one of these families of edges. If we delete the edges of $\mathcal{F}$, then each of the vertices has $k^2-k$ neighbours and the edges are coloured using exactly $k$ different colours. Let $H$ be this graph together with the edge-colouring. For example, in Figure \ref{fig:affinenoyellow} for $k=3$, $\mathcal{F}$ is the yellow family.
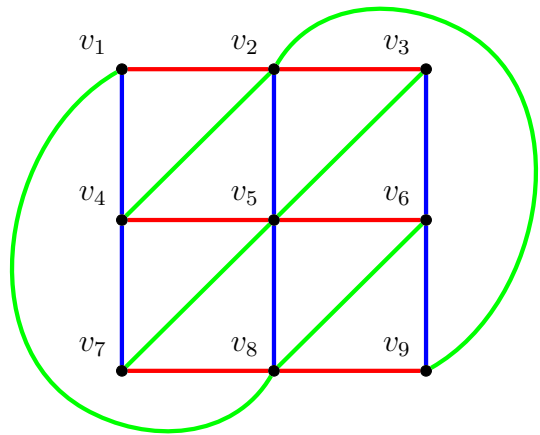
\begin{figure}[!htbp]
\centering
\begin{tikzpicture}
\node[main node, label={135:$v_1$}] (v1) at (-2,2) {};
\node[main node, label={135:$v_2$}] (v2) at (0,2) {};
\node[main node, label={135:$v_3$}] (v3) at (2,2) {};
\node[main node, label={135:$v_4$}] (v4) at (-2,0) {};
\node[main node, label={135:$v_5$}] (v5) at (0,0) {};
\node[main node, label={135:$v_6$}] (v6) at (2,0) {};
\node[main node, label={135:$v_7$}] (v7) at (-2,-2) {};
\node[main node, label={135:$v_8$}] (v8) at (0,-2) {};
\node[main node, label={135:$v_9$}] (v9) at (2,-2) {};

\draw[red, ultra thick] (v1) to (v2);
\draw[red, ultra thick] (v2) to (v3);
\draw[red, ultra thick] (v4) to (v5);
\draw[red, ultra thick] (v5) to (v6);
\draw[red, ultra thick] (v7) to (v8);
\draw[red, ultra thick] (v8) to (v9);

\draw[blue, ultra thick] (v1) to (v4);
\draw[blue, ultra thick] (v4) to (v7);
\draw[blue, ultra thick] (v2) to (v5);
\draw[blue, ultra thick] (v5) to (v8);
\draw[blue, ultra thick] (v3) to (v6);
\draw[blue, ultra thick] (v6) to (v9);

\draw[green, ultra thick] (v2) to (v4);
\draw[green, ultra thick] (v2) to[out=60,in=150] (2.5,2.5) to[out=-30,in=30](v9);
\draw[green, ultra thick] (v3) to (v5);
\draw[green, ultra thick] (v5) to (v7);
\draw[green, ultra thick] (v1) to[out=-150,in=150] (-2.5,-2.5) to[out=-30,in=-120] (v8);
\draw[green, ultra thick] (v6) to (v8);
\end{tikzpicture}
\caption{The graph $H$ when $k=3$ and $\mathcal{F}$ is the yellow family.}
\label{fig:affinenoyellow}
\end{figure}

We will construct the graph $G$ from $H$ by deleting a set $S$ of $k$ carefully chosen vertices and then blowing up the remaining vertices into vertex sets.

Pick any monochromatic component $C$ in $H$. The vertices of $C$ are $\{u_1,\dots,u_k\}$. We formed $H$ by deleting the edges of $\mathcal{F}$, a family of $k$ disjoint $K_k$, from $K_{k^2}$. Each $u_i\in C$ is in a different $K_k$ of $\mathcal{F}$. Let $w$ be another vertex in the same $K_k$ of $\mathcal{F}$ as $u_1$. We take $S$ to be the set $\{w,u_2,\dots,u_k\}$.

If we delete the vertices of $S$ from $H$, then there are $k^2-k$ vertices left and each of these has $(k-1)^2$ neighbours. We form the graph $G$ by blowing up this graph so each vertex set contains $\frac n{k^2-k}$ where $n$ is some multiple of $k^2-k$. We may colour the edges of $G$ as follows: the edges between vertex sets inherit their colour from the edges of $H$ and the edges within vertex sets are coloured arbitrarily.

Figure \ref{fig:deleteS} shows the case where $k=3$ and $\mathcal{F}$ is the yellow family. We choose the green component consisting of $v_1$, $v_6$ and $v_8$ as the basis for the set $S$ and choose $v_6$ to be $u_1$. As $v_2$ lies in the same yellow component of the affine plane as $v_6$, we choose $v_2$ to be $w$. Then $S=\{v_1, v_2, v_8\}$ and these are the vertices we remove from $H$ before blowing up the remaining vertices into vertex sets. All edges between vertex sets are present in Figure \ref{fig:deleteS}.
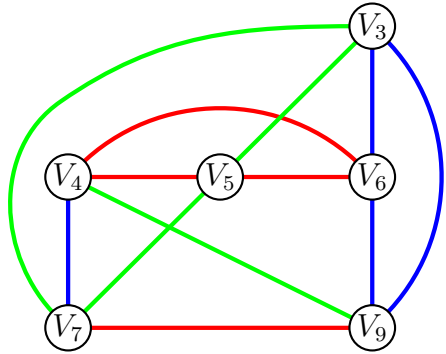
\begin{figure}[htbp]
\centering
\begin{tikzpicture}
\draw (2,2) node (v3) {};
\draw (-2,0) node (v4) {};
\draw (0,0) node (v5) {};
\draw (2,0) node (v6) {};
\draw (-2,-2) node (v7) {};
\draw (2,-2) node (v9) {};
	
\draw[red, ultra thick] (v4) to (v5);
\draw[red, ultra thick] (v5) to (v6);
\draw[red, ultra thick] (v4) to[out=45,in=135] (v6);
\draw[red, ultra thick] (v7) to (v9);

\draw[blue, ultra thick] (v4) to (v7);
\draw[blue, ultra thick] (v3) to (v6);
\draw[blue, ultra thick] (v6) to (v9);
\draw[blue, ultra thick] (v3) to[out=-45,in=45] (v9);

\draw[green, ultra thick] (v4) to (v9);
\draw[green, ultra thick] (v3) to (v5);
\draw[green, ultra thick] (v5) to (v7);
\draw[green, ultra thick] (v3) to[out=180,in=35] (-2.1,1) to[out=-145,in=135] (v7);

\draw[thick, fill=white] (v3) circle (0.3cm) node {$V_3$};
\draw[thick, fill=white] (v4) circle (0.3cm) node {$V_4$};
\draw[thick, fill=white] (v5) circle (0.3cm) node {$V_5$};
\draw[thick, fill=white] (v6) circle (0.3cm) node {$V_6$};
\draw[thick, fill=white] (v7) circle (0.3cm) node {$V_7$};
\draw[thick, fill=white] (v9) circle (0.3cm) node {$V_9$};
\end{tikzpicture}
\caption{The graph $G$ when $k=3$.}
\label{fig:deleteS}
\end{figure}

Each vertex in $G$ has degree $n(1-\frac{k-2}{k^2-k})-1$. Note that, if $n$ is sufficiently large, then $n(1-\frac{k-2}{k^2-k})-1\geq n(1-\frac{k-1}{k^2})$. Currently, the largest monochromatic component of $G$ has order exactly $\frac n{k-1}$. However, by slightly perturbing the sizes of the vertex sets, we can ensure that the largest monochromatic component of $G$ has strictly fewer than $\frac n{k-1}$ vertices.

Let $n=k(k-1)x+1$ where $x\geq2k$. In the graph $G$ described above, let the vertex sets have orders as follows:
\begin{itemize}
\item The vertex set corresponding to $u_1$ has order $x-(k-2)$
\item The $k-1$ vertex sets which correspond to the $k-1$ vertices that were in the same component of colour $k$ as $w$ have order $x+1$
\item All other vertex sets have order $x$. 
\end{itemize}
Figure \ref{fig:counterk} shows the perturbed graph $G$ for the case where $k=3$.
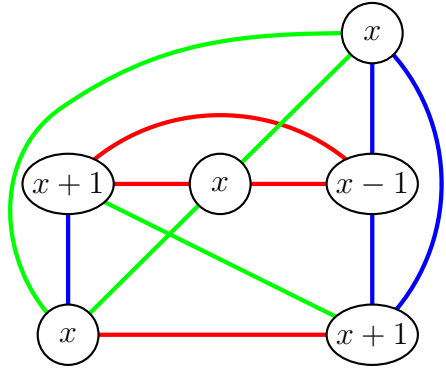
\begin{figure}[htbp]
\centering
\begin{tikzpicture}
\draw (2,2) node (v3) {};
\draw (-2,0) node (v4) {};
\draw (0,0) node (v5) {};
\draw (2,0) node (v6) {};
\draw (-2,-2) node (v7) {};
\draw (2,-2) node (v9) {};

\draw[red, ultra thick] (v4) to (v5);
\draw[red, ultra thick] (v5) to (v6);
\draw[red, ultra thick] (v4) to[out=45,in=135] (v6);
\draw[red, ultra thick] (v7) to (v9);

\draw[blue, ultra thick] (v4) to (v7);
\draw[blue, ultra thick] (v3) to (v6);
\draw[blue, ultra thick] (v6) to (v9);
\draw[blue, ultra thick] (v3) to[out=-45,in=45] (v9);

\draw[green, ultra thick] (v4) to (v9);
\draw[green, ultra thick] (v3) to (v5);
\draw[green, ultra thick] (v5) to (v7);
\draw[green, ultra thick] (v3) to[out=180,in=35] (-2.1,1) to[out=-145,in=135] (v7);

\draw[thick, fill=white] (v3) circle (0.4cm) node {$x$};
\draw[thick, fill=white] (v4) ellipse (0.6cm and 0.4cm) node {$x+1$};
\draw[thick, fill=white] (v5) circle (0.4cm) node {$x$};
\draw[thick, fill=white] (v6) ellipse (0.6cm and 0.4cm) node {$x-1$};
\draw[thick, fill=white] (v7) circle (0.4cm) node {$x$};
\draw[thick, fill=white] (v9) ellipse (0.6cm and 0.4cm) node {$x+1$};
\end{tikzpicture}
\caption{The perturbed graph $G$ when $k=3$.}
\label{fig:counterk}
\end{figure}

Each of the $k-1$ vertex sets of order $x+1$ was in a different clique of colour $k+1$ in the affine plane. Therefore, each monochromatic component of $G$ contains at most one of them, with the exception of the colour $k$ clique that contains all of them. This clique has order $(k-1)(x+1)\leq kx$ if $x\geq k-1$.

The only monochromatic cliques that contain $k$ vertex sets and a vertex set of order $x+1$ must also contain the vertex set of size $x-(k-2)$ and hence have order $kx-k-3\leq kx$. This is because the monochromatic component is not colour $k$ and so has to intersect with each of the cliques of colour $k$ in the affine plane. The only vertex set left in its colour $k$ clique is the one of order $x-(k-2)$. Hence every monochromatic component of $G$ has order at most $kx<\frac{n}{k-1}$.

The minimum degree of the graph $G$ is $(k-1)^2x+x-1$ because every vertex set is missing exactly $k-2$ other vertex sets and at most one of these has order $x+1$.  We further have $\delta(G)\geq (k^2-2k+2)x-1\ge \left(1-\frac{k-2}{k(k-1)}\right)n-2$ as required.

\section{Conclusion}
For any $n\in\mathbb{N}$ and $k\geq3$, let $f_k(n)$ be the maximum value such that there exists a graph $G$ on $n$ vertices with $\delta(G)=f_k(n)$ and a $k$-edge-colouring of $G$ where every monochromatic component has order strictly less than $\frac n{k-1}$. Theorem \ref{thm:56} implies that $f_3(n)<\left\lfloor\frac56n\right\rfloor$. In the proof of Theorem \ref{thm:counter3}, we found
\begin{align*}
f_3(n)&=\left\lfloor\frac56n\right\rfloor-1&\text{if }n&\equiv1,2,3,4,5\mod6\\
f_3(n)&\in\left\{\left\lfloor\frac56n\right\rfloor-2,\left\lfloor\frac56n\right\rfloor-1\right\}&\text{if }n&\equiv0\mod6.
\end{align*}
\noindent
We believe that $f_3(n)=\left\lfloor\frac56n\right\rfloor-1$ for all residue classes but we were unable to find an example when $n\equiv0\mod6$.

Theorems \ref{thm:56} and \ref{thm:counter3} prove that Conjecture \ref{conj:main} is false for $k=3$ and that the correct constant is $\frac56$. It is natural to ask what the correct bound is for other values of $k$. From the proof of Theorem \ref{thm:counterk}, we find that
\begin{equation*}
f_k(n)\geq n\left(1-\frac{k-2}{k(k-1)}\right)-2+\frac{k-2}{k(k-1)}\qquad\text{if }n\equiv1\mod k(k-1)
\end{equation*}
for infinitely many values of $k$.

We believe that the graphs constructed in the proof of Theorem \ref{thm:counterk} are close to optimal and hence we make the following conjecture.
\begin{conjecture}
\label{conj:new}
Fix $k\geq3$. Let $G$ be any graph with $n$ vertices and $\delta(G)\geq\left(1-\frac{k-2}{k(k-1)}\right)n$. If the edges of $G$ are $k$-coloured, then there exists a monochromatic component of order at least $\frac{n}{k-1}$.
\end{conjecture}

Although our methods for proving the upper bound extend in principle to $4$ or more colours, the computational time needed to run all of the required linear programs makes it infeasible to do so. For example, when $k=4$, a naive implementation of our approach would entail solving around $2^{240}$ linear programs. It would be nice to see whether Conjecture \ref{conj:new} is correct.

\subsection*{Note} This paper was submitted for review in September 2019. We later discovered that Rahimi \cite{R20} had independently proved Theorem \ref{thm:56}.

\subsection*{Acknowledgements}
We would like to thank Louis DeBiasio for his suggestions as to how a general counterexample to Conjecture \ref{conj:main} might be constructed. We would also like to thank the two anonymous referees for their helpful comments.

\appendix
\section{Implementing the linear programs}
\label{sec:lpcode}
The main obstacle in implementing Linear Programs \ref{lp:333} and \ref{lp:334} was the large number of possible intersection patterns ($2^{27}$ and $2^{36}$ respectively) that needed to be checked. We therefore used the implicit symmetry of the problem to reduce the number of linear programs which needed to be run.

Recall that the intersection pattern is given by the variables $(y_{ijk})$ where
\begin{equation*}
y_{ijk}=
\begin{cases}
1&\text{if }R_i\cap B_j\cap G_k\neq\emptyset\\
0&\text{otherwise.}
\end{cases}
\end{equation*}
We may assume that there are at least $3$ components of each colour (Lemma \ref{lem:2components}) and that each component intersects at least one component in each of the other colours (Lemma \ref{lem:allcolours}).

Given an intersection pattern $(y_{ijk})$, if there exists $i$ such that $y_{ijk}=0$ for all $j$ and all $k$, then this corresponds to the red component $R_i$ being empty (i.e there are at most two red components). It is therefore not necessary to run the linear program for this intersection pattern. (Indeed, the linear program has a value of $\alpha=\frac78$ which is optimal if we do not specify that there must be at least three components of each colour.)

We therefore excluded intersection patterns which corresponded to one of the components being empty. As this check only requires knowledge of the current intersection pattern, it is straightforward to do it when the intersection pattern has been generated.

There are two sources of symmetry in the problem: between components of the same colour and between components of different colours. Let us consider both.

Firstly, suppose we are given two intersection patterns, $(y_{ijk})$ and $(y'_{ijk})$. Suppose that, for some $t\in[2]$ and every $j$ and $k$, we have $y_{ijk}=y'_{(i+t)jk}$ where $i+t$ is calculated modulo $3$. Any optimal solution for $(y_{ijk})$ will also be an optimal solution for $(y'_{ijk})$ but with the red components relabelled. Therefore we only need to run the linear program for one of these intersection patterns to obtain the optimal solution for both. We can extend this idea to any intersection patterns which are the same up to relabelling of the components.

Secondly, if we have two intersection patterns $(y_{ijk})$ and $(y'_{ijk})$ such that, for every $k$, we have $y_{ijk}=y'_{jik}$, then any optimal solution for $(y_{ijk})$ will also be an optimal solution for $(y'_{ijk})$ but with the red and blue components swapped. Again we would only need to run the linear program for one of these intersection patterns to obtain the optimal solution for both. In the case where all three colours have exactly three components, all three colours are interchangeable; in the case where there are three red, three blue and four green components, only the red and blue components may be switched.

Unlike checking whether an intersection pattern corresponds to one of the components being empty, finding intersection patterns which are the same up to symmetry requires knowledge of both the current intersection pattern and other possible intersection patterns. Memory constraints make it impractical to generate all non-equivalent intersection patterns before running the linear programs. Instead, we consider only a subset of symmetries that we can handle efficiently using a version of lexicographic ordering.

First, for simplicity, suppose that we only have two colours, red and blue, and that the intersection matrix is given by $Z=(z_{ij})$ where $z_{ij}$ represents whether or not $R_i\cap B_j$ is empty. Swapping two rows in $Z$ corresponds to swapping the labels of two red components and similarly for columns and blue components. We define the \textit{lex value} of $Z$ to be
\begin{equation*}
\lex(Z)=\sum_{i=0}^2\sum_{j=0}^2z_{ij}100^{-i-j}.
\end{equation*}
Swapping pairs of rows and/or pairs of columns in $Z$ can change its lex value. Configurations with more entries that are 1 in the top left corner of the matrix will give higher lex values that those where the top left corner contains many 0s.

By only swapping pairs of rows or pairs of columns that strictly increase the value of $\lex(Z)$, we will eventually reach $Z'$, a configuration of $Z$ where the lex value is $\maxlex(Z)$, the unique maximum possible lex value of $Z$. Both $Z$ and $Z'$ are possible intersection patterns and, because we obtained $Z'$ from $Z$ through a series of row and column swaps, $Z$ and $Z'$ are equivalent intersection patterns.

We only run a linear program on $Z$ if $\lex(Z)=\maxlex(Z)$. This significantly reduces the number of linear programs that need to be run whilst still ensuring that at least one linear program is run for each class of equivalent intersection patterns.

Now consider the situation we actually have where the intersection pattern is given by $Y=(y_{ijk})$. Whilst we could extend the definition of lex value to a three-dimensional matrix, it proved cumbersome to calculate the maximum lex. Instead, we calculated $\lex(Y^{(k)})$ where $Y^{(k)}$ is the $3\times3$ matrix obtained by restricting to a fixed value of $k$. We ran the linear program on $Y$ if $\lex(Y^{(k)})\geq\lex(Y^{(k+1)})$ for every $k$ and $\lex(Y^{(1)})=\maxlex(Y^{(1)})$. This method eliminated sufficiently many intersection patterns to make the computation tractable.
\end{document}